\documentclass[11pt]{amsart}

\usepackage[margin=3cm]{geometry}

\usepackage[foot]{amsaddr}
\usepackage{amsmath}
\usepackage{amssymb}
\usepackage{amsthm}
\usepackage{graphicx}
\usepackage{float}
\usepackage{color}
\usepackage{dsfont}

\newtheorem{theorem}{Theorem}[section]
\newtheorem{corollary}{Corollary}

\newtheorem{lemma}[theorem]{Lemma}
\newtheorem{proposition}{Proposition}

\theoremstyle{definition}

\newtheorem{remark}{Remark}

\title[Monotonicity of principal eigenvalue] 
      {Monotonicity of principal eigenvalue for elliptic operators with incompressible flow: A functional approach}

\author[SHUANG LIU, YUAN LOU]{}



 \email{liushuangnqkg@ruc.edu.cn; lou@math.ohio-state.edu}

\begin{document}
\maketitle

\centerline{\scshape Shuang Liu }
\medskip
{\footnotesize
 \centerline{Institute for Mathematical Sciences, Renmin University of China}
   \centerline{Beijing 100872, PR China}
} 

\bigskip

\centerline{\scshape Yuan Lou}
\medskip
{\footnotesize
\centerline{Institute for Mathematical Sciences, Renmin University of China}
\centerline{Beijing 100872, PR China;}
   \centerline{Department of Mathematics, Ohio State University}
   \centerline{ Columbus, OH 43210, USA}
}

\bigskip


\begin{abstract}
We establish the monotonicity 
of the principal eigenvalue $\lambda_1(A)$,  as a function of the advection amplitude $A$,
 for the elliptic operator $L_{A}=-\mathrm{div}(a(x)\nabla )+A\mathbf{V}\cdot\nabla +c(x)$ with
   incompressible flow $\mathbf{V}$,
    subject to Dirichlet, Robin and Neumann boundary conditions.
    As a consequence, the limit of $\lambda_1(A)$ as $A\to \infty$
    always exists and is finite for  Robin boundary conditions. These
    results  answer some open questions raised by Berestycki, Hamel and Nadirashvili \cite{nk2}.
    Our method relies upon some  functional which is associated with principal eigenfuntions
    for operator  $L_A$  and its adjoint operator.
    As a byproduct of the approach,  a new min-max characterization of $\lambda_1(A)$ is given.
\end{abstract}

\section{\bf Introduction}\label{S1}

There have been extensive studies on the reaction-diffusion equations
of the form
\begin{equation}\label{Intro1}
    w_t=\mathrm{div}(a(x)\nabla w)-A\mathbf{V}\cdot\nabla w+wf(x,w),
\end{equation}
which model various physical, chemical, and biological processes:
On unbounded domains \cite{nk7,k0}, compact manifolds \cite{k1},
and bounded domains with appropriate boundary conditions \cite{nk0,nk2, CC2003, Ni2011}. Let $\Omega$ be a bounded region of $\mathbb{R}^{N}$ with smooth
boundary $\partial\Omega$, and $\mathbf{n}(x)$ be the outward unit normal vector at $x\in\partial\Omega$.  Consider equation (\ref{Intro1})
defined on $\Omega$ and suppose that $w$ satisfies $bw+(1-b)[a(x)\nabla w] \cdot\mathbf{n}=0$ on $\partial\Omega$.
 The stability of steady state $w\equiv0$ is determined  by the sign of the principal eigenvalue, denoted as $\lambda_1(A)$, for the
 linear eigenvalue problem
$$
L_Au:=-\mathrm{div}(a(x)\nabla u)+A\mathbf{V}\cdot\nabla u+c(x)u=\lambda_1(A)u,
$$
subject to boundary conditions $bu+(1-b)[a(x)\nabla u] \cdot\mathbf{n}=0$ on $\partial\Omega$, where $c(x)=-f(x,0)$,
 and parameter $b\in [0, 1]$.

Of particular interest is  the dependence of the principal eigenvalue $\lambda_1(A)$ on the advection amplitude $A$.
If vector field  $\mathbf{V}$ is  incompressible, i.e.,  $\mathrm{div}\mathbf{V}=0$ in $\Omega$,
  Berestycki et al.
  investigated in \cite{nk2} the asymptotic behavior of  $\lambda_1(A)$ as $A$ approaches infinity,  and
  they identified a direct link between the limit of $\lambda_1(A)$ and the first integral set of $\mathbf{V}$,
   defined as
\begin{equation*}
\mathcal{I}_{b}=
 \begin{cases}
 \begin{split}
 &\{\varphi \in  H^{1}(\Omega):\varphi\neq0, \mathbf{V}\cdot\nabla \varphi=0~\mathrm{a.e.}~\mathrm{in}~\Omega\}, \quad ~~\mathrm{}~0\leq b<1\\
 &\{\varphi \in  H_{0}^{1}(\Omega):\varphi\neq0, \mathbf{V}\cdot\nabla \varphi=0~\mathrm{a.e.}~\mathrm{in}~\Omega\}, \quad  ~~~~~~\mathrm{}~b=1.
  \end{split}
 \end{cases}
 \end{equation*}
More precisely, Berestycki et al. showed  in \cite{nk2}
 that for the operator $L_A$ defined on $\Omega$  with  Dirichlet ($b=1$) or Neumann ($ b=0$) boundary conditions,
 $\lambda_1(A)$ stays bounded as $A\rightarrow+\infty$ if and only if  $\mathcal{I}_1\neq\varnothing$ or $\mathcal{I}_0\neq\varnothing$, respectively.
 Furthermore, they proved that for any $A\ge 0$,
\begin{equation}\label{Liu4}
 \begin{split}
 \lambda_1(0)\le \lambda_1(A)\le \lim_{A\rightarrow+\infty}\lambda_1(A)=\inf_{\omega\in\mathcal{I}_{0}~\mathrm{or}~\mathcal{I}_{1}}\frac{\int_{\Omega}\nabla\omega\cdot[a(x)\nabla \omega]\mathrm{d}x+\int_{\Omega}c(x)\omega^2\mathrm{d}x}{\int_{\Omega}\omega^2\mathrm{d}x}.
  \end{split}
 \end{equation}
 That is, $\lambda_1(A)$ attains its minimum at $A=0$ and  its maximum at $A=\infty$.
 As mentioned in \cite{nk2}, $\lambda_1(A)$ is a nondecreasing function of $|A|$ if  $\mathbf{V}$ is
  an incompressible gradient flow.
   Nevertheless, this monotonicity property has remained  open for a general incompressible flow $\mathbf{V}$.

   The primary goal of this paper is to answer the above open question affirmatively.
To this end, we shall focus on the following eigenvalue problem with a general incompressible flow $\mathbf{V}$,
 subject to general boundary conditions:
\begin{equation}\label{Liu1}
 \begin{cases}
 \begin{split}
 &L_A u_A=-\mathrm{div}(a(x)\nabla u_A)+A\mathbf{V}\cdot\nabla u_A+c(x)u_A=\lambda_1(A)u_A~~\ \ \mathrm{in}~\Omega,\\
 &u_A>0~~\ \ \mathrm{in}~\Omega,\\
  &bu_A+(1-b)[a(x)\nabla u_A] \cdot\mathbf{n}=0~~\ \ \mathrm{on}~\partial\Omega.
  \end{split}
 \end{cases}
 \end{equation}

Throughout this paper we always assume that $c\in C^{\alpha}(\bar{\Omega})$ and the diffusion matrix $a(x)$ is symmetric and uniformly  elliptic $C^{1,\alpha}(\bar{\Omega})$ matrix field
satisfying
$$\exists \ 0<\gamma_1<\gamma_2,~\mathrm{such}~\mathrm{that}~\gamma_1|\xi|^{2}\leq\xi^{\mathrm{T}}a(x)\xi\leq\gamma_2|\xi|^{2},\forall x\in\Omega, \forall \xi\in\mathbb{R}^{N},
$$
for some constant $\alpha\in(0,1)$. Furthermore, we always assume that
the vector field $\mathbf{V}\in C^{1}(\bar{\Omega})$ satisfying $\mathrm{div}\mathbf{V}=0$ in $\Omega$,
whereas an additional assumption stating that $\mathbf{V}\cdot \mathbf{n}=0~\mathrm{on}~\partial\Omega$
is always assumed for the case of $0\leq b<1$.
Under these assumptions the Krein-Rutman Theorem  guarantees the existence of the principle eigenvalue $\lambda_1(A)$
 and it can be easily shown that $\lambda_1(A)$
  is  symmetric in $A$. Therefore, throughout this paper we shall assume $A\geq0$.

Our first result can be stated as follows.
\begin{theorem}\label{th3}
Let  $L_A$ be the elliptic operator defined by {\rm(\ref{Liu1})}
and $\lambda_1(A)$ be its principle eigenvalue.   Then
the following statements hold:

\medskip
\noindent{$~~~~\mathrm{(i)}$}  If $ u_0\not\in\mathcal{I}_{b}$,
then $\frac{\partial\lambda_1}{\partial A}(A)>0$ for every $A>0$;

\smallskip

\noindent{$\mathrm{(ii)}$} If $ u_0\in\mathcal{I}_{b}$, then $\lambda_1(A)\equiv\lambda_1(0)$ for every $A>0$.

\medskip

\noindent{} Here $u_0$ is the principal eigenfunction of $L_0$ satisfying
\begin{equation*}
 \begin{cases}
 \begin{split}
 &-\mathrm{div}(a(x)\nabla u_0)+c(x)u_0=\lambda_1(0)u_0~~\ \ \mathrm{in}~\Omega,\\
 &u_0>0~~\ \ \mathrm{in}~\Omega,\\
  &bu_0+(1-b)[a(x)\nabla u_0]\cdot\mathbf{n}=0~~\  \ \mathrm{on}~\partial\Omega.
  \end{split}
 \end{cases}
 \end{equation*}
\end{theorem}

Theorem \ref{th3} implies that the strict monotonicity of  $\lambda_1(A)$ with respect to
 the advection amplitude $A$ relies on $u_0$, the principal eigenfunction of
 operator $L_0$. Interpreting this in the context of convection-enhanced diffusion, Theorem \ref{th3}
  suggests that larger advection amplitude generally produces faster mixing for reaction-diffusion-advection
   equation (\ref{Intro1}) as long as $ u_0\not\in\mathcal{I}_{b}$.
    In this sense, Theorem \ref{th3}
      seems to refine the well-known statement that mixing by an incompressible flow enhances diffusion
      in various contexts \cite{k1, nk7,H1, HN1, nk4,kiselev, nk10, k0,k011}.

 Our next result, as a corollary of Theorem \ref{th3}, provides the boundedness and
asymptotic behavior of $\lambda_1(A)$ for Robin boundary conditions, consistent with the main result in  \cite{nk2} for Neumann boundary
conditions. 
\begin{theorem}\label{Rth2}
If $0\leq b<1$,
 the limit  $\lim_{A\to+\infty}\lambda_1(A)$ always exists, is finite and satisfies
\begin{equation}\label{Liu2}
 \begin{split}
 \lim_{A\rightarrow+\infty}\lambda_1(A)\leq\inf_{\omega\in\mathcal{I}_{b}}\frac{\frac{b}{1-b}\int_{\partial\Omega}\omega^2\mathrm{d}S_x+\int_{\Omega}\nabla\omega\cdot[a(x)\nabla \omega]\mathrm{d}x+\int_{\Omega}c(x)\omega^2\mathrm{d}x}{\int_{\Omega}\omega^2\mathrm{d}x}.
  \end{split}
 \end{equation}
In particular, the principal eigenvalues $\lambda_1(A)$  of {\rm(\ref{Liu1})} are uniformly bounded.
\end{theorem}

The proof of the boundedness for $\lambda_1(A)$  in Theorem \ref{Rth2}
 is essentially due to  Berestycki et al. \cite{nk2}. 
 Nevertheless, the existence of the limit $\lim_{A\to\infty}\lambda_1(A)$ for Robin boundary
 conditions appears to be new.

The proof of Theorem \ref{th3} relies heavily  on properties of  certain functional. 
 Set $L:=-\mathrm{div}(a(x)\nabla )+\mathbf{V}\cdot\nabla +c(x)$,
   with adjoint operator  $L^{*}:=-\mathrm{div}(a(x)\nabla )-\mathbf{V}\cdot\nabla +c(x)$,
    in view of $\mathrm{div}\mathbf{V}=0$ in $\Omega$
     and particularly $\mathbf{V}\cdot \mathbf{n}=0~\mathrm{on}~\partial\Omega$ for case $0\le b<1$.
     By $u,v$ we further denote the normalized principal eigenfunctions corresponding to $L$ and $L^{*}$,
     respectively. %
      In terms of operator $L$ and $u$, $v$, we now introduce  functional $J$,
\begin{equation*}
    J(\omega)=\int_{\Omega}u v\left(\frac{L\omega}{\omega}\right)\mathrm{d}x,
\end{equation*}
which is well defined on the cone
\begin{equation*}
\mathbb{S}_b=
 \begin{cases}
 \begin{split}
 &\{\varphi \in  C^{2}(\Omega)\cap C^1(\bar{\Omega}):\varphi >0 ~\mathrm{in}~\Omega ~\mathrm{,}~~ b\varphi +(1-b)[a(x)\nabla \varphi] \cdot\mathbf{n}=0~~\mathrm{on}~\partial\Omega\},~~\mathrm{for}~0\leq b<1\\
 &\{\varphi \in  C^{2}(\Omega)\cap C^1(\bar{\Omega}):\varphi >0 ~\mathrm{in}~\Omega, ~\varphi=0~\mathrm{on}~\partial\Omega~~\mathrm{,}~~ \nabla\varphi\cdot\mathbf{n}<0~\mathrm{on}~\partial\Omega\},~~~~~~\mathrm{for}~b=1.
  \end{split}
 \end{cases}
 \end{equation*}

 A direct observation from the definition of  functional $J$  leads to $J(u)=\lambda_1$ and a
  far less obvious result (see Lemma \ref{L2}) says
   that functional $J$ attains its maximum at the principal eigenfunction
  $u$ and its scalar multiples. This is crucial to the proof of Theorem \ref{th3} and it also allows us to explore a
   new min-max characterization of the principal eigenvalue.

The characterization of  the principal eigenvalue  has always been an interesting and active topic,
and we refer to Donsker and Varadhan, Nussbaum and Pinchover for some earlier works
  \cite{nk6,nk8,nk20}.
Employing the maximum principle,  Protter and Weinberger \cite{nk21} established a classical characterization
of the principal eigenvalue for
general second order elliptic operators $P$,  given by the min-max formula
\begin{equation}\label{Intro3}
    \lambda_1=\sup_{\omega>0}\inf_{x\in\Omega}\left[\frac{P\omega(x)}{\omega(x)}\right].
\end{equation}
This characterization is valid for general
elliptic operators in both bounded and unbounded domains \cite{nk20,nk21}.
As a byproduct of properties of functional $J$, we have the following characterization for $\lambda_1$:
\begin{theorem}\label{th4}For elliptic operator $L$ with an incompressible flow $\mathbf{V}$ subject to
general boundary conditions with $0\leq b\leq1$, the principal eigenvalue $\lambda_1$ can be characterized as
\begin{equation}\label{Liu14}
\begin{split}
  \lambda_1=\inf_{p\in \mathbb{S}_b,\int_{\Omega}p^2=1}\sup_{\omega\in\mathbb{S}_b}\int_{\Omega}p^2(x)\left(\frac{L\omega}{\omega}\right)\mathrm{d}x.
\end{split}
\end{equation}
\end{theorem}
This min-max formula may not be
valid for general second elliptic operators,
 and it reduces to the classical Rayleigh-Ritz formula when $V=0$, by treating $p^2\mathrm{d}x$ as some probability measure;
 See Remark \ref{R3} for details. Different from the formula (\ref{Intro3}), the  min-max characterization in Theorem \ref{th4}
 relies on the properties of functional $J$.  They however may be connected via a min-max theorem in \cite{nk24}.
 Via  functional $J$ we observe that the min-max  formula attains the extremum when $p^2=uv$.

The rest of this paper is organized as follows: In Section \ref{S2}, we shall give some
properties of functional $J$. Section \ref{S3} is devoted to the proof of Theorems \ref{th3}
 and  \ref{Rth2}. In Section \ref{S4} we establish the new  min-max characterization of the principal eigenvalue.
  Finally, the implications of our method/results and some open questions will be discussed in Section \ref{S5}.

\section{\bf Properties of functional $J$ }\label{S2}
 We shall present some properties of functional $J$ in this section, which are crucial to the proofs of main results in this paper. Before proceeding further, we point out again that throughout this paper, $u$ and $v$ are the principal eigenfunctions corresponding to $L$ and $L^{*}$, respectively,  with general  boundary conditions.
  Due to the slight difference between  the definitions of  functional $J$ in the cases of $0\leq b<1$ and $b=1$, we divide this section into two subsections.

 \subsection{Neumann and Robin boundary conditions: $0\leq b<1$}\label{sub1}
 Recalling the regularity requirements of coefficients $c$, $\mathbf{V}$ and matrix field $a(x)$, Sobolev embedding theorem implies that $u, v\in C^{2,\alpha}(\Omega)$ and $u,v\in \mathbb{S}_b$ for $0\leq b<1$. We emphasize here that the constant $b$ is confined to $0\leq b<1$  unless otherwise specified, and the incompressible flow $\mathbf{V}$ satisfies $\mathrm{div}\mathbf{V}=0$ in $\Omega$ with $\mathbf{V}\cdot \mathbf{n}=0~\mathrm{on}~\partial\Omega$ in this subsection. Also, the eigenfunctions can be normalized as $\int_{\Omega}u^2\mathrm{d}x=1$ and $\int_{\Omega}uv\mathrm{d}x=1$. We now  recall the functional associated to
 operator $L$ with Neumann or Robin boundary conditions, defined on $\mathbb{S}_b$ as in Section \ref{S1},
 \begin{equation}\label{Liu3}
   J(\omega)=\int_{\Omega}u v\left(\frac{L\omega}{\omega}\right)\mathrm{d}x,~~\quad \mathrm{}~\omega\in\mathbb{S}_b.
 \end{equation}
For any $\omega\in \mathbb{S}_b$,  a simple but useful observation from (\ref{Liu3}) leads to
\begin{equation}\label{16}
\begin{split}
J(\omega)=&-\int_{\Omega}u v\left[\frac{\mathrm{div}(a(x)\nabla\omega)}{\omega}\right]\mathrm{d}x+\int_{\Omega} u v\left[\frac{\mathbf{V}\cdot\nabla\omega}{\omega}\right]\mathrm{d}x+\int_{\Omega}u v c\,\mathrm{d}x\\
=&-\int_{\partial\Omega}u v\Big[a(x)\nabla\log\omega\Big]\cdot\mathbf{n}\mathrm{d}S_x+\int_{\Omega}\nabla\left(\frac{u v}{\omega}\right)\cdot \Big[a(x)\nabla\omega\Big]\mathrm{d}x\\
&+\int_{\Omega} u v\mathbf{V}\cdot\nabla\log\omega\mathrm{d}x+\int_{\Omega}u v c\mathrm{d}x\\
=&-\int_{\partial\Omega}u v\Big[a(x)\nabla\log\omega\Big]\cdot\mathbf{n}\mathrm{d}S_x-\int_{\Omega}u v\Big\{(\nabla\log\omega)\cdot\left[ a(x)\nabla\log\omega\right]\Big\}\mathrm{d}x\\
&+\int_{\Omega}\Big[ u v\mathbf{V}+a(x)\nabla(u v)\Big]\cdot\nabla\log\omega\mathrm{d}x+\int_{\Omega}u v c\mathrm{d}x.\\
\end{split}
\end{equation}



By equality (\ref{16}), 
we show that the principal eigenfunction $u$ is  a critical point of $J$.
\begin{proposition}\label{P1}
$J'(u)\varphi=0$ for all $\varphi\in  \tilde{\mathbb{S}}_b\triangleq\Big\{\varphi\in C^{2}(\Omega)\cap C^1(\bar{\Omega}): b\varphi+(1-b)\left[a(x)\nabla \varphi\right]\cdot\mathbf{n}=0~\mathrm{on}~\partial\Omega\Big\}$.
\end{proposition}
\begin{proof}
Using equality (\ref{16}), the  Fr\'{e}chet derivation $J'(\omega)$ of $\omega\in\mathbb{S}_b$ can be written as
\begin{equation}\label{17}
\begin{split}
J'(\omega)\varphi=&-\int_{\partial\Omega}u v\left[a(x)\nabla\left(\frac{\varphi}{\omega}\right)\right]\cdot\mathbf{n}\mathrm{d}S_x-2\int_{\Omega}u v\left\{(\nabla\log\omega)\cdot\left[ a(x)\nabla\left(\frac{\varphi}{\omega}\right)\right]\right\}\mathrm{d}x\\
&+\int_{\Omega}\Big[ u v\mathbf{V}+a(x)\nabla(u v)\Big]\cdot\nabla\left(\frac{\varphi}{\omega}\right)\mathrm{d}x,
\end{split}
\end{equation}
for all $\varphi\in \tilde{\mathbb{S}}_b$.
 By the boundary conditions of $u$ and $v$, a direct calculation via integration by parts gives
\begin{equation*}
\begin{split}
&J'(u)\varphi\\
=&-\int_{\partial\Omega}u v\left[a(x)\nabla\left(\frac{\varphi}{u}\right)\right]\cdot\mathbf{n}\mathrm{d}S_x-2\int_{\Omega}u v\left\{(\nabla\log u)\cdot\left[ a(x)\nabla\left(\frac{\varphi}{u}\right)\right]\right\}\mathrm{d}x\\
&+\int_{\Omega}\Big[ u v\mathbf{V}+a(x)\nabla(u v)\Big]\cdot\nabla\left(\frac{\varphi}{u}\right)\mathrm{d}x\\
=&-\int_{\partial\Omega}u v\left[a(x)\nabla\left(\frac{\varphi}{u}\right)\right]\cdot\mathbf{n}\mathrm{d}S_x-2\int_{\partial\Omega}\left(\frac{v\varphi}{u}\right)\Big[ a(x)\nabla u\Big]\cdot\mathbf{n}\mathrm{d}S_x+\int_{\partial\Omega}\left(\frac{\varphi}{u}\right)\Big[a(x)\nabla (uv)\Big]\cdot\mathbf{n}\mathrm{d}S_x\\
&+2\int_{\Omega}\left(\frac{\varphi}{u}\right)\nabla\cdot\Big[v a(x)\nabla u\Big]\mathrm{d}x-\int_{\Omega}\left(\frac{\varphi}{u}\right)\nabla\cdot\Big[ u v\mathbf{V}+a(x)\nabla(u v)\Big]\mathrm{d}x\\
=&-\int_{\partial\Omega}v\Big[a(x)\nabla\varphi\Big]\cdot\mathbf{n}\mathrm{d}S_x+\int_{\partial\Omega}\varphi\Big[a(x)\nabla v\Big]\cdot\mathbf{n}\mathrm{d}S_x+2\int_{\Omega}\left(\frac{\varphi}{u}\right)\Big\{v\mathrm{div}(a(x)\nabla u)+\nabla v\cdot [a(x)\nabla u]\Big\}\mathrm{d}x\\
&-\int_{\Omega}\left(\frac{\varphi}{u}\right)\Big[\nabla(u v)\cdot\mathbf{V}\Big]\mathrm{d}x-\int_{\Omega}\left(\frac{\varphi}{u}\right)\mathrm{div}\Big[a(x)\nabla(u v)\Big]\mathrm{d}x\\
=&2\int_{\Omega}\left(\frac{\varphi}{u}\right)\Big\{v\mathrm{div}(a(x)\nabla u)+\nabla v\cdot [a(x)\nabla u]\Big\}\mathrm{d}x-\int_{\Omega}\mathbf{V}\cdot\left[\nabla v+v\left(\frac{\nabla u}{u}\right)\right]\varphi\mathrm{d}x\\
&-\int_{\Omega}\left(\frac{\varphi}{u}\right)\Big\{v\mathrm{div}(a(x)\nabla u)+2\nabla v\cdot [a(x)\nabla u]+u\mathrm{div}(a(x)\nabla v)\Big\}\mathrm{d}x\\
=&\int_{\Omega}\left(\frac{v\varphi}{u}\right)\mathrm{div}(a(x)\nabla u)\mathrm{d}x-\int_{\Omega} v\varphi\mathbf{V}\cdot\left(\frac{\nabla u}{u}\right)\mathrm{d}x-\int_{\Omega} \varphi\mathbf{V}\cdot\nabla v\mathrm{d}x-\int_{\Omega}\varphi\mathrm{div}(a(x)\nabla v)\mathrm{d}x\\
=&-\int_{\Omega}\left(\frac{v}{u}\right)\Big[-\mathrm{div}(a(x)\nabla u)+\mathbf{V}\cdot\nabla u\Big]\varphi\mathrm{d}x+\int_{\Omega}\Big[-\mathrm{div}(a(x)\nabla v)-\mathbf{V}\cdot\nabla v\Big]\varphi\mathrm{d}x.
\end{split}
\end{equation*}
Here we used the additional assumption  $\mathbf{V}\cdot\mathbf{n}=0~\mathrm{on}~\partial\Omega$ and the boundary conditions of $v$ and $\varphi$ to remove the boundary integral. Recall the fact that $L u=\lambda_1u$ and $L^{*} v=\lambda_1v$ and proceed to compute
\begin{equation*}
\begin{split}
J'(u)\varphi=&-\int_{\Omega}\left(\frac{v}{u}\right)\left(\lambda_1u-c u\right)\varphi\mathrm{d}x+\int_{\Omega}\left(\lambda_1v-c v\right)\varphi\mathrm{d}x=0,
\end{split}
\end{equation*}
as anticipated. The proof is complete.
\end{proof}

Next we establish a crucial property of functional $J$.
\begin{lemma}\label{L2}
For any  $\omega\in\mathbb{S}_b$, the following formula holds:
$$J(u)=J(\omega)+\int_{\Omega}u v\left\{\left[\nabla\log\left(\frac{\omega}{u}\right)\right]\cdot\left[ a(x)\nabla\log\left(\frac{\omega}{u}\right)\right]\right\}\mathrm{d}x.$$
\end{lemma}
\begin{proof}

To obtain this formula, some elementary but
 a bit tedious manipulations are needed. Together with equality (\ref{16}), a direct  calculation  yields
\begin{equation*}
\begin{split}
& J(u)-J(\omega)\\
=&\int_{\partial\Omega}u v\left[a(x)\nabla\log\left(\frac{\omega}{u}\right)\right]\cdot\mathbf{n}\mathrm{d}S_x+\int_{\Omega}u v\Big\{(\nabla\log\omega)\cdot\left[ a(x)\nabla\log\omega\right]\Big\}\mathrm{d}x\\
&-\int_{\Omega}u v\Big\{(\nabla\log u)\cdot\left[ a(x)\nabla\log u\right]\Big\}\mathrm{d}x-\int_{\Omega}\Big[ u v\mathbf{V}+a(x)\nabla(u v)\Big]\cdot\nabla\log\left(\frac{\omega}{u}\right)\mathrm{d}x\\
=&\int_{\partial\Omega}u v\left[a(x)\nabla\log\left(\frac{\omega}{u}\right)\right]\cdot\mathbf{n}\mathrm{d}S_x+\int_{\Omega}u v\left\{\Big[\nabla\log\left(u \omega\right)\Big]\cdot \left[ a(x)\nabla\log\left(\frac{\omega}{u}\right)\right]\right\}\mathrm{d}x\\
&-\int_{\Omega}\Big[ u v\mathbf{V}+a(x)\nabla(u v)\Big]\cdot\nabla\log\left(\frac{\omega}{u}\right)\mathrm{d}x\\
=&\int_{\partial\Omega}u v\left[a(x)\nabla\log\left(\frac{\omega}{u}\right)\right]\cdot\mathbf{n}\mathrm{d}S_x+\int_{\Omega}u v\left\{\left[\nabla\log\left(\frac{\omega}{u}\right)\mathrm{d}x+2\nabla\log u\right]\cdot\left[ a(x)\nabla\log\left(\frac{\omega}{u}\right)\right]\right\}\mathrm{d}x\\
&-\int_{\Omega}\Big[ u v\mathbf{V}+a(x)\nabla(u v)\Big]\cdot\nabla\log\left(\frac{\omega}{u}\right)\mathrm{d}x\\
=&\int_{\Omega}u v\left\{\left[\nabla\log\left(\frac{\omega}{u}\right)\right] \cdot \left[a(x)\nabla\log\left(\frac{\omega}{u}\right)\right]\right\}\mathrm{d}x+\int_{\partial\Omega}u v\left[a(x)\nabla\log\left(\frac{\omega}{u}\right)\right]\cdot\mathbf{n}\mathrm{d}S_x\\
&+2\int_{\Omega}u v\left\{(\nabla\log u)\cdot\left[ a(x)\nabla\log\left(\frac{\omega}{u}\right)\right]\right\}\mathrm{d}x-\int_{\Omega}\Big[ u v\mathbf{V}+a(x)\nabla(u v)\Big]\cdot\nabla\log\left(\frac{\omega}{u}\right)\mathrm{d}x,\\
\end{split}
\end{equation*}
where we have used the symmetry of matrix field $a(x)$
and the boundary conditions of $\omega$ and $u$.  By straightforward calculations
we have $u\log\left(\frac{\omega}{u}\right)\in \tilde{\mathbb{S}}_b$ for any $\omega\in \mathbb{S}_b$. Choosing $\varphi=u\log\left(\frac{\omega}{u}\right)$ in equality (\ref{17}), by Proposition \ref{P1} we have
\begin{equation*}
\begin{split}
   J(u)-J(\omega)&=\int_{\Omega}u v\left\{\left[\nabla\log\left(\frac{\omega}{u}\right)\right]\cdot\left[ a(x)\nabla\log\left(\frac{\omega}{u}\right)\right]\right\}\mathrm{d}x-J'(u)\varphi\\
   &=\int_{\Omega}u v\left\{\left[\nabla\log\left(\frac{\omega}{u}\right)\right]\cdot\left[ a(x)\nabla\log\left(\frac{\omega}{u}\right)\right]\right\}\mathrm{d}x.
   \end{split}
\end{equation*}
The assertion of Lemma \ref{L2} thus follows.
\end{proof}


The following result is an immediate consequence of Lemma  \ref{L2}.
\begin{corollary}\label{C1}
$$\int_{\Omega}v L u\mathrm{d}x-\int_{\Omega}u L v\mathrm{d}x=\int_{\Omega}u v\left\{\left[\nabla\log\left(\frac{v}{u}\right)\right]\cdot\left[ a(x)\nabla\log\left(\frac{v}{u}\right)\right]\right\}\mathrm{d}x.$$
\end{corollary}
\begin{proof}
A simple observation leads to
$$\int_{\Omega}u L v\mathrm{d}x=\int_{\Omega}u v\left(\frac{L v}{v}\right)\mathrm{d}x=J(v),$$
and analogously  $\int_{\Omega}v L u\mathrm{d}x=J(u)$. Hence Corollary \ref{C1} follows from  Lemma \ref{L2}.
\end{proof}

\subsection{Dirichlet boundary conditions: $b=1$}\label{sub2}
The case of Dirichlet boundary conditions is slightly different from the Neumann or Robin
 boundary conditions, as noted in \cite{nk2}. It is perhaps worth pointing out that in this case, the functional $J$ shall be defined on $\mathbb{S}_1$ and the extra assumption $\mathbf{V}\cdot n=0~\mathrm{on}~\partial\Omega$ is not needed
  for further discussions.  Hopf Boundary Lemma implies that $\nabla u\cdot\mathbf{n}<0$ and $\nabla v\cdot\mathbf{n}<0$ on $\partial\Omega$, and thus  $u,v\in \mathbb{S}_1$ so that  $J(u),J(v)$ are well defined. Moreover, the adjoint operator of $L$ subject to Dirichlet boundary conditions can be written as $L^{*}=-\mathrm{div}(A(x)\nabla )-\mathbf{V}\cdot\nabla +c(x)$ without $\mathbf{V}\cdot n=0~\mathrm{on}~\partial\Omega$, due to $u=0$ on $\partial\Omega$. Thanks to $\nabla\omega\cdot\mathbf{n}<0$ on $\partial\Omega$,  we have $\frac{uv}{\omega}=0$ on $\partial\Omega$ to get $\int_{\partial\Omega}u v\left[a(x)\nabla\log\omega\right]\cdot\mathbf{n}\mathrm{d}S_x=0$ in equality (\ref{16}).

With the same argument as in the Neumann or Robin boundary conditions,
getting rid of all boundary
integrals, we can show that the principal eigenfunction $u$ is still a critical point of $J$ in this case,
 i.e., $J'(u)\varphi=0$ for all $\varphi\in\tilde{\mathbb{S}}_1$. Based on this fact, the formula in Lemma \ref{L2}
  remains true. As the proof is similar, thus it is omitted. Therefore, the properties of functional $J$ listed
   in subsection \ref{sub1} hold for all $0\leq b\leq1$.

\section{\bf Monotonicity and boundedness of principal eigenvalue}\label{S3}
 Recall that $L_{A}=-\mathrm{div}(a(x)\nabla )+A\mathbf{V}\cdot\nabla +c(x)$ and its adjoint operator $L^{*}_{A}=-\mathrm{div}(a(x)\nabla )-A\mathbf{V}\cdot\nabla +c(x)$.
 Here we emphasize that throughout this paper,   $\mathbf{V}$ satisfies $\mathrm{div}\mathbf{V}=0~\mathrm{in}~\Omega$ and an additional assumption $\mathbf{V}\cdot n=0~\mathrm{on}~\partial\Omega$ is also needed for $0\leq b<1$ (see Remark \ref{R2} below). For all $A\geq0$, there exists a unique principal eigenvalue $\lambda_1(A)$
 for eigenvalue problem (\ref{Liu1}), and a unique (up to multiplication) eigenfunction $u_A$ satisfying problem (\ref{Liu1}). 
 We also denote the principle eigenfunction of $L^{*}_{A}$ by some normalized positive function $v_A$
 and write the functional related with problem (\ref{Liu1}) as
$$
J_A(\omega)=\int_{\Omega}u_A v_A\left(\frac{L_A\omega}{\omega}\right)\mathrm{d}x,
\quad\omega\in\mathbb{S}_b.
 $$

Our first goal of this section is to show Theorem \ref{th3}.

\medskip

\noindent{}$\mathrm{\mathbf{Proof~of ~Theorem~\ref{th3}.}}$
Firstly,  if $u_0\in\mathcal{I}_{b}$, then for every $A>0$,  $u_0$ satisfies
\begin{equation*}
 \begin{cases}
 \begin{split}
 &-\mathrm{div}(a(x)\nabla u_0)+A\mathbf{V}\cdot\nabla u_0+c(x)u_0=\lambda_1(0)u_0~~\quad \mathrm{in}~\Omega,\\
 &u_0>0~~\quad \mathrm{in}~\Omega,\\
  &bu_0+(1-b)[a(x)\nabla u_0]\cdot\mathbf{n}=0~~\quad \mathrm{on}~\partial\Omega.
  \end{split}
 \end{cases}
 \end{equation*}
 Hence, $\lambda_1(A)=\lambda_1(0)$ for all $A>0$. This proves part (i).

For the proof of  part (ii), we assume that $u_0\not\in\mathcal{I}_{b}$.
We normalize $u_A$ and $v_A$ such that $\int_\Omega u_A^2\,dx=\int_\Omega u_A v_A\,dx=1$.

Differentiate equation (\ref{Liu1}) with respect to $A$ and denote $\frac{\partial u_A}{\partial A}= u'_A$ for
the sake of brevity, we obtain
\begin{equation}\label{11}
 \begin{cases}
 \begin{split}
 &-\mathrm{div}\Big[a(x)\nabla_x u'_A\Big]+A\mathbf{V}\cdot\nabla_x u'_A+\mathbf{V}\cdot\nabla_x u_A+c(x)u'_A=\frac{\partial\lambda_1}{\partial A}(A)u_A+\lambda_1(A)u'_A~~\mathrm{in}~\Omega,\\
  &bu'_A+(1-b)[a(x)\nabla_x u'_A]\cdot\mathbf{n}=0~~\mathrm{on}~\partial\Omega,~~~\quad \mathrm{} ~~\int_{\Omega}u'_Au_A\mathrm{d}x=0.
  \end{split}
 \end{cases}
 \end{equation}
Multiply (\ref{11}) by $v_A$ and integrate the result in $\Omega$,
 together with the definition of $v_A$ we have
\begin{equation}\label{12}
   \frac{\partial\lambda_1}{\partial A}(A)=\int_{\Omega}v_A\mathbf{V}\cdot\nabla u_A\mathrm{d}x.
\end{equation}

Observe that $u_0=v_0$ for $A=0$. This leads to
$$
\frac{\partial\lambda_1}{\partial A}(0)=\frac{1}{2}\int_{\Omega}\mathbf{V}\cdot\nabla u_0^{2}\mathrm{d}x=0.
$$
Here we used that   $\mathbf{V}$ is divergence free together with $\mathbf{V}\cdot n=0$ on $\partial\Omega$ for $0\leq b<1$ and $u_0=0$ on $\partial\Omega$ for $b=1$.

\medskip

\noindent{}$\mathrm{\mathbf{Claim}}:$ For each $A>0$, $\frac{\partial\lambda_1}{\partial A}(A)\ge 0$,
and either $\frac{\partial\lambda_1}{\partial A}(A)>0$, or $\lambda_1(A)=\lambda_1(0)$.\\

To establish this assertion, it is illuminating to consider the special case of $A=1$. Recall the definition of $L_1$ and $L^{*}_1$ to rewrite equality (\ref{12}) as
$$\frac{\partial\lambda_1}{\partial A}(1)=\frac{1}{2}\int_{\Omega}v_1(L_1-L^{*}_1)u_1\mathrm{d}x=\frac{1}{2}\left[\int_{\Omega}v_1L_1u_1\mathrm{d}x-\int_{\Omega}u_1L_1v_1\mathrm{d}x\right].$$
A direct application of Corollary \ref{C1} and positive definiteness of $a(x)$ yields
$$\frac{\partial\lambda_1}{\partial A}(1)=\frac{1}{2}\int_{\Omega}u_1 v_1\left\{\left[\nabla\log\left(\frac{v_1}{u_1}\right)\right]\cdot\left[ a(x)\nabla\log\left(\frac{v_1}{u_1}\right)\right]\right\}\mathrm{d}x\geq0,$$
and $\frac{\partial\lambda_1}{\partial A}(1)=0$ if and only if $u_1=cv_1$ for some $c>0$. By  $\int_{\Omega}u^2_1=1$ and $\int_{\Omega}u_1v_1=1$, we  see that $c=1$ and $u_1=v_1$.  Furthermore, if $u_1=v_1$, thus
 $L_1u_1=L^{*}_1u_1=\lambda_1(1)u_1$ and hence $\mathbf{V}\cdot\nabla u_1=0$, which further implies that
\begin{equation*}
 \begin{cases}
 \begin{split}
 &-\mathrm{div}(a(x)\nabla u_1)+c(x)u_1=\lambda_1(1)u_1~~\quad \mathrm{in}~\Omega,\\
 &u_1>0~~\quad \mathrm{in}~\Omega,\\
  &bu_1+(1-b)[a(x)\nabla u_1]\cdot\mathbf{n}=0~~\quad \mathrm{on}~\partial\Omega.
  \end{split}
 \end{cases}
 \end{equation*}
Hence, $\lambda_1(1)=\lambda_1(0)$. In summary,  $\frac{\partial\lambda_1}{\partial A}(1)\ge 0$,
and either $\frac{\partial\lambda_1}{\partial A}(1)>0$, or $\lambda_1(1)=\lambda_1(0)$.

We now proceed to consider the general case of $A>0$. Rewrite the operator $L_A$ as
\begin{equation*}
   L_{A}=A\Big(-\mathrm{div}(a(x)\nabla)+\mathbf{V}\cdot\nabla+c(x)\Big)+(1-A)(-\mathrm{div}(a(x)\nabla)+c(x))=A L_1+ (1-A)L_0
\end{equation*}
and define a new elliptic operator $\mathcal{L}_B$ by 
$$ \mathcal{L}_B:=B L_A+ (1-B)L_0.$$
It is easy to verify that $\mathcal{L}_B=AB L_1+ (1-AB)L_0=L_{AB}$. Set $r_1(B)$ as the principal eigenvalue of  $\mathcal{L}_B$. A natural fact is that $r_1(B)=\lambda_1(AB)$. Similar to the above discussion for $B=1$, it follows that $\frac{\partial r_1}{\partial B}(1)\ge 0$, and either $\frac{\partial r_1}{\partial B}(1)>0$, or $r_1(1)=r_1(0)$. In view of $\frac{\partial r_1}{\partial B}(1)=A\frac{\partial\lambda_1}{\partial A}(A)$, the Claim is proved.\\

Before proceeding further to show $\frac{\partial\lambda_1}{\partial A}(A)>0$ for
 all $A>0$, let us calculate $\frac{\partial^{2}\lambda_1}{\partial A^{2}}(0)$ firstly. Differentiate equation (\ref{11}) with respect to $A$ again, and applying the notation $\frac{\partial^{2} u_A}{\partial A^{2}}=u''_A$ for brevity arrives at
\begin{equation}\label{13}
 \begin{cases}
 \begin{split}
 &-\mathrm{div}\Big[a(x)\nabla_x u''_A\Big]+A\mathbf{V}\cdot\nabla_x u''_A+2\mathbf{V}\cdot\nabla_x u'_A+c(x)u''_A\\
 &~~~~=\frac{\partial^{2}\lambda_1}{\partial A^{2}}(A)u_A+2\frac{\partial\lambda_1}{\partial A}(A)u'_A+\lambda_1(A)u''_A~~\quad \mathrm{in}~\Omega,\\
  &bu''_A+(1-b)[a(x)\nabla_x u''_A]\cdot\mathbf{n}=0~~\quad \mathrm{on}~\partial\Omega.
  \end{split}
 \end{cases}
 \end{equation}
 Setting $A=0$ in (\ref{13})
  and multiplying it
   by $u_0$ and integrating the result in $\Omega$, it follows from $\frac{\partial\lambda_1}{\partial A}(0)=0$ that
\begin{equation*}
    \frac{\partial^{2}\lambda_1}{\partial A^{2}}(0)=2\int_{\Omega}u_0\mathbf{V}\cdot\nabla_x u'_0\mathrm{d}x.
\end{equation*}
On the other hand, multiplying equation (\ref{11}) by $u'_0$ and setting $A=0$, we have
\begin{equation*}
\begin{split}
   &\frac{b}{1-b}\int_{\partial\Omega}(u'_0)^2\mathrm{d}S_x+\int_{\Omega}\nabla_x u'_0\cdot \Big[ a(x)\nabla_x u'_0\Big]\mathrm{d}x-\int_{\Omega}u_0\mathbf{V}\cdot\nabla_x u'_0\mathrm{d}x+\int_{\Omega}c(x)(u'_0)^{2}\mathrm{d}x
   \\
   &=\lambda_1(0)\int_{\Omega}(u'_0)^{2}\mathrm{d}x,
\end{split}
\end{equation*}
which in turn implies that
\begin{equation}\label{18}
\begin{split}
    \frac{1}{2}\frac{\partial^{2}\lambda_1}{\partial A^{2}}(0)
    =&\frac{b}{1-b}\int_{\partial\Omega}(u'_0)^2\mathrm{d}S_x+\int_{\Omega}\nabla_x u'_0\cdot \Big[ a(x)\nabla_x u'_0\Big]\mathrm{d}x+\int_{\Omega}c(x)(u'_0)^{2}\mathrm{d}x
    \\
    &-\lambda_1(0)\int_{\Omega}(u'_0)^{2}\mathrm{d}x.
\end{split}
\end{equation}

We are now in a position  to prove Theorem \ref{th3}. According to the above Claim, it suffices
to prove that $\lambda_1(A)>\lambda_1(0)$ for every $A>0$. If $\lambda_1(\hat{A})=\lambda_1(0)$ for
some  $\hat{A}>0$, since $\frac{\partial\lambda_1}{\partial A}(A)\ge 0$,
 $\lambda_1(A)\equiv\lambda_1(0)$ for $A\in [0, \hat{A}]$.
  Thus $\frac{\partial^{2}\lambda_1}{\partial A^{2}}(0)=0$. By  (\ref{18}) we have
 $$\lambda_1(0)=\frac{ \frac{b}{1-b}\int_{\partial\Omega}(u'_0)^2\mathrm{d}S_x+\int_{\Omega}\nabla_x u'_0\cdot [ a(x)\nabla_x u'_0]\mathrm{d}x+\int_{\Omega}c(x)(u'_0)^{2}\mathrm{d}x}{\int_{\Omega}(u'_0)^{2}\mathrm{d}x},$$
 so the variational argument of principal eigenvalue $\lambda_1(0)$ implies that $ u'_0=cu_0$ for some constant $c$. Setting $A=0$ and then substituting equality $ u'_0=cu_0$ into equation (\ref{11}), we can conclude that $\mathbf{V}\cdot\nabla u_0\equiv0$ in $\Omega$,
  which is a contradiction. 
  This completes the proof. \qed\\

We now proceed to prove Theorem \ref{Rth2}. \\

\noindent{}$\mathrm{\mathbf{Proof~of ~Theorem~\ref{Rth2}.}}$ It suffices to establish the following result: \\

\noindent{}$\mathrm{\mathbf{Claim~1.}}$ Assume that $\mathcal{I}_{b}\neq\varnothing$. Then $\lambda_1(A)$ is uniformly bounded and
\begin{equation*}
 \begin{split}
 \lambda_1(A)\leq\inf_{\omega\in\mathcal{I}_{b}}\frac{\frac{b}{1-b}\int_{\partial\Omega}\omega^2\mathrm{d}S_x+\int_{\Omega}\nabla\omega\cdot[a(x)\nabla \omega]\mathrm{d}x+\int_{\Omega}c(x)\omega^2\mathrm{d}x}{\int_{\Omega}\omega^2\mathrm{d}x},~~\quad \forall A\geq0.
  \end{split}
 \end{equation*}

 The idea of the proof for Claim 1 comes from Theorem 2.2 in \cite{nk2} and we shall sketch the proof for
 the sake of completeness. Note that $u_A>0$ in $\bar{\Omega}$ by Hopf Boundary Lemma for case of $0\leq b<1$. Choose any function $\omega\in\mathcal{I}_{b}$ and multiply the equation of $u_A$ 
  by $\frac{\omega^2}{u_A}$, then integration  by parts implies that
 \begin{equation}\label{Re3}
  \begin{split}
    &\frac{b}{1-b}\int_{\partial\Omega}\omega^2\mathrm{d}S_x+\int_{\Omega}\nabla\left(\frac{\omega^2}{u_A}\right)\cdot\Big[a(x)\nabla u_A\Big]\mathrm{d}x+A\int_{\Omega}\omega^2\mathbf{V}\cdot\nabla \log u_A\mathrm{d}x+\int_{\Omega}c\omega^2\mathrm{d}x\\
    &=\lambda_1(A)\int_{\Omega}\omega^2\mathrm{d}x.
    \end{split}
\end{equation}
An interesting observation, in analogy with the proof of Theorem 2.2 in \cite{nk2}, gives that
 \begin{equation*}
   \int_{\Omega}\omega^2\mathbf{V}\cdot\nabla \log u_A\mathrm{d}x=0~\mathrm{and }~ \int_{\Omega}\nabla\left(\frac{\omega^2}{u_A}\right)\cdot\Big[a(x)\nabla u_A\Big]\mathrm{d}x\leq \int_{\Omega}\nabla\omega\cdot[a(x)\nabla \omega]\mathrm{d}x,
\end{equation*}
 which leads to Claim 1 by combining equality (\ref{Re3}) and  $\mathcal{I}_{b}\neq\varnothing$. 

It turns out that $\mathcal{I}_{b}\neq\varnothing$ always holds for $0\leq b<1$, since it at least follows that $c\in\mathcal{I}_{b}$ for any constant $c$. Together with  Claim 1,  the monotonicity of $\lambda_1(A)$ in Theorem \ref{th3} readily implies that the limit of $\lim_{A\to\infty} \lambda_1(A)$ always exists and is finite. The proof of Theorem \ref{Rth2} is complete.  \qed\\

\begin{remark}\label{R2}(Necessity of the assumption $\mathbf{V}\cdot \mathbf{n}=0~\mathrm{on}~\partial\Omega$): We now remark that the additional assumption $\mathbf{V}\cdot \mathbf{n}=0~\mathrm{on}~\partial\Omega$ is necessary for $0\leq b<1$, while not necessary for $b=1$, corresponding to zero Dirichlet boundary condition.
\begin{itemize}
  \item  For $b=1$,  zero Dirichlet boundary condition implies $u_A=v_A=0$ on $\partial\Omega$ and the adjoint operator of $L_A$ can be written as $L^{*}_A=-\mathrm{div}(a(x)\nabla )-A\mathbf{V}\cdot\nabla +c(x)$ without the additional assumption, whence Theorem \ref{th3} remains true as the  properties of $J_A$ in Section \ref{S2} hold without this assumption as stated in subsection \ref{sub2}.
  \item For $0\leq b<1$, Theorem \ref{th3} may fail without the assumption $\mathbf{V}\cdot\mathbf{n}=0~\mathrm{on}~\partial\Omega$. Consider the same example as in Remark 2.5 of \cite{nk2},
\begin{equation*}
 \begin{cases}
 \begin{split}
 &-\varphi''_{A}+A\varphi'_{A}+c(x)\varphi_{A}=\lambda_1(A)\varphi_{A},~~0<x<1,\\
  &\varphi'_{A}(0)=\varphi'_{A}(1)=0.
  \end{split}
 \end{cases}
 \end{equation*}
Here we consider the special case where $b=0$ and the incompressible flow $\mathbf{V}=1$ does not satisfy the assumption $\mathbf{V}\cdot\mathbf{n}=0$ at 0 and 1. Chen and Lou's result in \cite{nk3} implies $\lim_{A\rightarrow+\infty}\lambda_1(A)=c(0)$ by treating $\mathbf{V}=-\nabla(-x)$. Assume further that $c'(x)\geq0$ and $c(x)\not\equiv\mathrm{constant}$. If Theorem \ref{th3} holds, since $\lambda_1(0)\geq\min_{x\in[0,1]}c(x)=c(0)$, we have $\lambda_1(A)\equiv c(0)$,  and thus $\varphi'_0=0$ according to part (ii) in  Theorem \ref{th3}, which contradicts to $c(x)\not\equiv\mathrm{constant}$.
\end{itemize}
\end{remark}


\section{\bf Min-Max characterization of principal eigenvalue}\label{S4}
In this section  we focus on a new min-max characterization of the principal eigenvalue for elliptic operator
$L=-\mathrm{div}(a(x)\nabla)+\mathbf{V}\cdot\nabla +c(x)$ with incompressible flow and
general  boundary conditions.
To state our main result, some preparations are needed. In this connection, in view of the classical min-max characterization of principal eigenvalue \cite{nk21}
$$\lambda_1=\sup_{\omega\in\mathbb{S}_b}\inf_{x\in\Omega}\left[\frac{L\omega(x)}{\omega(x)}\right]=\inf_{\omega\in\mathbb{S}_b}\sup_{x\in\Omega}\left[\frac{L\omega(x)}{\omega(x)}\right]$$
together with the facts
\begin{equation*}
    \inf_{p\in \mathbb{S}_b,\int_{\Omega}p^2=1}\int_{\Omega}p^2(x)\left(\frac{L\omega}{\omega}\right)
    \mathrm{d}x=\inf_{x\in\Omega}\left[\frac{L\omega(x)}{\omega(x)}\right],\\
\end{equation*}
and
\begin{equation*}
   \sup_{p\in \mathbb{S}_b,\int_{\Omega}p^2=1}\int_{\Omega}p^2(x)\left(\frac{L\omega}{\omega}\right)\mathrm{d}x=\sup_{x\in\Omega}\left[\frac{L\omega(x)}{\omega(x)}\right],\\
\end{equation*}
it is straightforward to derive the following min-max characterization  of $\lambda_1$:
\begin{equation}\label{14}
\begin{split}
  \lambda_1&=\sup_{\omega\in\mathbb{S}_b}\inf_{p\in \mathbb{S}_b,\int_{\Omega}p^2=1}\int_{\Omega}p^2(x)\left(\frac{L\omega}{\omega}\right)\mathrm{d}x
  \\
  &=\inf_{\omega\in\mathbb{S}_b}\sup_{p\in \mathbb{S}_b,\int_{\Omega}p^2=1}\int_{\Omega}p^2(x)\left(\frac{L\omega}{\omega}\right)\mathrm{d}x.
\end{split}
\end{equation}

 However, the min-max characterization in Theorem \ref{th4} is somewhat different.
The following result is the key of the proof of  Theorem \ref{th4}:

\begin{lemma}\label{L3} $$\sup_{\omega\in\mathbb{S}_b}J(\omega)=J(u)=\lambda_1.$$
Furthermore, if $J(\omega_0)=\sup_{\omega\in\mathbb{S}_b}J(\omega)$ for some $\omega_0\in\mathbb{S}_b$, then  $\omega_0=cu$ for some constant $c>0$.
\end{lemma}
Lemma \ref{L3} is a direct consequence of Lemma \ref{L2} by recalling the  positive definiteness of $a(x)$.
With the help of Lemma \ref{L3}, Theorem \ref{th4} can be proved in  straightforward manner
as follows.\\

\noindent{}$\mathrm{\mathbf{Proof~of ~Theorem~\ref{th4}.}}$
We first  choose $p^2=u v $ and apply Lemma \ref{L3} to obtain  that
$$\lambda _1=\sup_{\omega\in\mathbb{S}_b}\int_{\Omega}u  v \left(\frac{L\omega}{\omega}\right)\mathrm{d}x\geq \inf_{p\in \mathbb{S}_b,\int_{\Omega}p^2=1}\sup_{\omega\in\mathbb{S}_b}\int_{\Omega}p^2(x)\left(\frac{L\omega}{\omega}\right)\mathrm{d}x.$$
On the other hand, for any $p\in \mathbb{S}_b$ satisfying $\int_{\Omega}p^2=1$, it is easy to see that
$$\lambda _1=\int_{\Omega}p^2(x)\left(\frac{L u }{u }\right)\mathrm{d}x\leq\sup_{\omega\in\mathbb{S}_b}\int_{\Omega}p^2(x)\left(\frac{L\omega}{\omega}\right)\mathrm{d}x,$$
which implies that
$$\lambda _1
\leq
\inf_{p\in \mathbb{S}_b,\int_{\Omega}p^2=1}\sup_{\omega\in\mathbb{S}_b}\int_{\Omega}p^2(x)\left(\frac{L\omega}{\omega}\right)\mathrm{d}x.
$$
Hence  equality (\ref{Liu14}) holds. 
The proof of Theorem \ref{th4} is now complete. \qed

\begin{remark}\label{R3}(Reduce to the classical Rayleigh-Ritz formula): 
The classical Rayleigh-Ritz formula is actually implicity contained in the min-max formula in Theorem \ref{th4} if $L$ is self-adjoint, i.e.,  $\mathbf{V}=0$. It can be deduced from an important result in \cite{nk25}. More specifically, viewing $\mu=p^2\mathrm{d}x$ as a positive measure satisfying the mild assumption $\mu\ll\lambda$ for the Borel measure $\lambda$ and noting that $\frac{\mathrm{d}\mu}{\mathrm{d}\lambda}=p^2$, Theorem 4 in \cite{nk25} leads to
$$\sup_{\omega\in\mathbb{S}_b}\int_{\Omega}p^2(x)\left(\frac{L\omega}{\omega}\right)\mathrm{d}x= \langle Lp, p \rangle,$$
which reduces the formula in Theorem \ref{th4} to the classical Rayleigh-Ritz formula.
\end{remark}

\section{\bf Discussions and open questions}\label{S5}


In many physical and biological systems, the effect of  incompressible flow $\mathbf{V}$ on  the
speed of traveling fronts  of equation (\ref{Intro1}) remains an important area of active research \cite{nk27, nk15, nk26, NX2003, NX2005, NX2007, NRX2015, nk14},  with particular interest on  the minimal speed $c^{*}_\mathbf{V}$.
  The minimal speed $c^{*}_\mathbf{V}$ can be enhanced by the introduction of incompressible flows  
  \cite{nk12,k1,nk7,nk13, k0}, while  general compressible flows may decrease $c^{*}_\mathbf{V}$; See Theorem 2.8 of \cite{nk16}. In this connection, many works focus on the case of the shear flow $\mathbf{V}=\alpha(x_2,\ldots,x_N)\mathbf{e}$, where $\alpha\not\equiv0$ is zero-average, in a straight cylinder $\Omega=\mathbb{R}\times D$ with  bounded domain $D\subset\mathbb{R}^{N-1}$ along the direction $\mathbf{e}$. Examples are known for which  the minimal speed $c^{*}_{A\mathbf{V}}$,
  in the presence of a shear flow  $\mathbf{V}$, is asymptotically linear in $A$ \cite{nk26}.
  Furthermore, $c^{*}_{A\mathbf{V}}$ is increasing in $A$, ${c^{*}_{A\mathbf{V}}}/{A}$  is decreasing in $A$,
   as well as ${c^{*}_{A\mathbf{V}}}/{A}\rightarrow\rho>0$ as $A\rightarrow+\infty$  \cite{nk27, nk16}. The monotonicity of $c^{*}_{A\mathbf{V}}$ and ${c^{*}_{A\mathbf{V}}}/{A}$ however remains open for  general incompressible flow $\mathbf{V}$; See  Remark 1.9 in \cite{nk12} and Remark 1.6 in \cite{nk26} for details.
  Our preliminary studies suggest that the  monotonicity of ${c^{*}_{A\mathbf{V}}}/{A}$ holds
  for general incompressible flow $\mathbf{V}$.
We hope to report it in forthcoming work.

 We now turn to consider operator $L_A$ with gradient  flow $\mathbf{V}_1=\nabla m$ for some $m\in C^2(\bar{\Omega})$,
 where the principal eigenvalue $\lambda_1(A)$, in analogy with equation (1.2) in \cite{nk2}, can be written as
\begin{equation*}
 \begin{split}
 \lambda_1(A)=\inf_{\omega\in H^1(\Omega)\setminus\{0\}}
\frac{\frac{b}{1-b}\int_{\partial\Omega}\omega^2\mathrm{d}S_x+\int_{\Omega}\nabla \omega\cdot[a(x)\nabla \omega]\mathrm{d}x+\int_{\Omega}\left(\frac{A^2}{4}|\mathbf{V}_1|^2-\frac{A}{2}\mathrm{div}\mathbf{V}_1+c(x)\right)\omega^2\mathrm{d}x}
{\int_{\Omega}\omega^2\mathrm{d}x},
  \end{split}
 \end{equation*}
which implies the monotonicity of  $\lambda_1(A)$ if $\mathbf{V}_1$ is incompressible satisfying $\mathrm{div}\mathbf{V}_1=0$. This result can be covered by Theorem \ref{th3} with the extra assumption $\mathbf{V}_1\cdot\mathbf{n}=0$ on $\partial\Omega$.
However, if the gradient flow $\mathbf{V}_1=\nabla m$ is incompressible and satisfies $\mathbf{V}_1\cdot\mathbf{n}=0$ on $\partial\Omega$, the only possibility is $m=\mathrm{constant}$. Hence we may ask naturally:
When does the monotonicity  property remain true for gradient flow?
 Understanding the monotonicity of $\lambda_1(A)$ with general flows seems to be more difficult.

 Another open question is to determine the limit value of  $\lambda_1(A)$ for incompressible flow $\mathbf{V}$ with
  Robin boundary conditions as $A\rightarrow+\infty$, though the existence of the limit has been shown in Theorem \ref{Rth2}. The results for  Dirichlet and Neumann boundary conditions in \cite{nk2}
   show that the limit of $\lambda_1(A)$ can be  determined by the variational principle (\ref{Liu4}).
    In view of Theorem \ref{Rth2}, it seems plausible to conjecture that for $0\leq b<1$,
\begin{equation*}
 \begin{split}
 \lim_{A\rightarrow+\infty}\lambda_1(A)=\inf_{\omega\in\mathcal{I}_{b}}\frac{\frac{b}{1-b}\int_{\partial\Omega}\omega^2\mathrm{d}S_x+\int_{\Omega}\nabla\omega\cdot[a(x)\nabla \omega]\mathrm{d}x+\int_{\Omega}c(x)\omega^2\mathrm{d}x}{\int_{\Omega}\omega^2\mathrm{d}x},
  \end{split}
 \end{equation*}
which would reduce to the results in \cite{nk2} for the case $b=0$.  
The  limit value of  $\lambda_1(A)$ with the gradient flow $\mathbf{V}_1=\nabla m$ has been established by  Chen and Lou \cite{nk3} for Neumann boundary conditions, which can be stated as
$$\lim_{A\rightarrow+\infty}\lambda_1(A)=\min_{\mathcal{M}}c,$$
with the set $\mathcal{M}$  consisting of all points of local maximum of $m$. Hence a natural question arises: Does the limit of  $\lambda_1(A)$ exist as $A\rightarrow+\infty$ for general flows under proper boundary conditions? If it exists, what is the limit value?

There are a substantial body of literatures concerning the asymptotic behavior of the principal
eigenvalue of elliptic operators for small diffusion rates; See
\cite{chenlou2012, DEF1, DEF2, Fri1, Wen1}. For the principal eigenvalue of operator $L_{D}=-D\Delta+\mathbf{V}\cdot\nabla +c(x)$,
Chen and Lou \cite{chenlou2012} investigated its asymptotic behavior as $D\to 0$ when $\mathbf{V}$
is a gradient flow. Much less seems to be known
when $\mathbf{V}$ is a general incompressible flow; See \cite{BZ2017, Vu1}.

\bigskip

\bigskip

\noindent{\bf Acknowledgments.} SL was partially supported by the NSFC grant No. 11571364.
YL  was partially supported by the NSF grant DMS-1411176.


\end{document}